\documentclass{svmult}

\usepackage{makeidx}         
\usepackage{graphicx}        
\usepackage{multicol}        
\usepackage{amsmath}        
\usepackage{ntheorem}		
\theoremseparator{.}
\usepackage{bbding}		
\usepackage{amssymb}        
\usepackage[bottom]{footmisc}
\usepackage{url}

\theorembodyfont{\normalfont}
\renewtheorem{conjecture}{Conjecture}
\newcommand{\forceqed}{\begin{flushright}\vspace{-1.85em}\Square\end{flushright}}

\DeclareMathOperator{\DefEq}{\stackrel{{\rm def}}{=}}
\DeclareMathOperator{\DefEqDisp}{\hspace{0.15em}\DefEq\hspace{0.15em}}

\makeindex             

\begin{document}

\title*{Extending Babbage's (Non-)Primality Tests}
\titlerunning{Extending Babbage's (Non-)Primality Tests}
\author{Jonathan Sondow}
\institute{209 West 97th Street, New York, NY 10025 \texttt{jsondow@alumni.princeton.edu}
}
\maketitle

\begin{abstract}\mbox{We recall Charles Babbage's $1819$ criterion for primality, based} on simultaneous congruences for binomial coefficients, and extend it to a least-prime-factor test. We also prove a partial converse of his non-primality test, based on a single congruence. Two problems are posed.
Along the way we encounter Bachet, Bernoulli, B\'{e}zout, Euler, Fermat, Kummer, Lagrange, Lucas, Vandermonde, Waring, Wilson, Wolstenholme, 
and several contemporary mathematicians.
\end{abstract}

\section{Introduction} \label{SEC:intro}

Charles Babbage was an English mathematician, philosopher, inventor, mechanical engineer, and ``irascible genius'' who pioneered computing machines \cite{babbagebio,beyer,grabiner,moseley,mactutor,odonnell}.
Although he held the Lucasian Chair of Mathematics at Cambridge University from $1828$ to $1839$, 
during that period he never resided in Cambridge or delivered a lecture \cite{blackwood}, \cite[p.~$7$]{dubbey}.\\

 \centerline{\includegraphics[width=.9in]{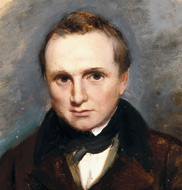}} \vspace{-.3 cm} \begin{center}{Charles Babbage (1791--1871)}\end{center}

In $1819$ he published his only work on number theory, a short paper \cite{babbage} that begins:

\begin{quotation}
The singular theorem of Wilson respecting Prime Numbers, which was first
published by Waring in his {\em Meditationes Analyticae} \cite[p.~218]{waring}, and to which neither himself
nor its author could supply the demonstration, excited the attention of the
most celebrated analysts of the continent, and to the labors of Lagrange \cite{lagrange} and Euler
we are indebted for several modes of proof $\dotso.$
\end{quotation}
Babbage formulated {\bf Wilson's theorem} as a criterion for primality:
{\em an integer $p>1$ is a prime if and only if $(p-1)! \equiv -1 \!\pmod{p}$}. (For a modern proof, see Moll \cite[p.~66]{moll}.) 
He  then introduced several such criteria, involving congruences for binomial coefficients (see Granville \cite[Sections~1 and 4]{granville}). However, some of his claims were unproven or even wrong
(as Dubbey points out in \cite[pp.~139--141]{dubbey}). One of his valid results is a necessary and sufficient condition for primality, based on a number of simultaneous congruences. 
Henceforth let $n$ denote an integer.

\begin{theorem}[Babbage's Primality Test] \label{THM:test}
An integer $p>1$ is a prime if and only if 
\begin{equation} \label{EQ:test}
\binom{p+n}{n}\equiv 1 \pmod{p} 
\end{equation}
for all $n$ satisfying $0 \le n \le p-1.$
\end{theorem}

This is of only theoretical interest, the test being slower than trial division.

The ``only if'' part is an immediate consequence of the beautiful {\bf theorem of Lucas} \cite{lucas} (see \cite{fine,granville,mestrovic0,mestrovic2} and \cite[p.~70]{moll}), which
asserts that {\em if $p$ is a prime and the non-negative integers
  $a = \alpha_0 + \alpha_1 \, p + \dotsb + \alpha_r \, p^r$ and
  \mbox{$b = \beta_0 + \beta_1 \, p + \dotsb + \beta_r \, p^r$}
are written in base $p$ $($so that \mbox{$0 \le \alpha_i, \beta_i \le p-1$} for all~$i)$, then}
\begin{equation}  \label{EQ:Lucas}
  \binom{a}{b} \equiv
\prod_{i=0}^r  \binom{\alpha_i}{\beta_i} \pmod{p}.
\end{equation}
(Here the convention is that $\binom{\alpha}{\beta} =0$ if $\alpha < \beta$.) The congruence \eqref{EQ:test} follows if \mbox{$0 \le n \le p-1$}, for then all the binomial coefficients formed on the right-hand side of \eqref{EQ:Lucas} are of the form $\binom{\alpha}{\alpha} =1,$ except the last one, which is $\binom{1}{0} =1.$

However, the theorem was not available to Babbage, because when it was published in $1878$ he had been dead for seven years.

Lucas's theorem implies more generally that {\em for $p$ a prime and $m$ a power of $p,$ the congruences
\begin{equation} \label{EQ:test2}
\binom{m+n}{n}\equiv 1 \pmod{p} \qquad (0 \le n \le m-1)
\end{equation}
hold.} A converse was proven
in $2013$:
{\bf Me\v{s}trovi\'{c}'s theorem} \cite{mestrovic2} states that {\em if $m>1$ and $p>1$ are integers such that \eqref{EQ:test2} holds,
then $p$ is a prime and $m$ is a power of~$p.$}
To begin the proof, Me\v{s}trovi\'{c} noted that for $n = 1$ the hypothesis gives
\begin{equation*} \label{EQ:Mestpf}
\binom{m+1}{1} = m+1 \equiv 1 \pmod{p}\quad \implies\quad p\mid m.
\end{equation*}
The rest of the proof involves combinatorial congruences modulo prime powers.

As Me\v{s}trovi\'{c} pointed out,
\noindent 
``the `if' part of Theorem~\ref{THM:test} is an immediate consequence of [his theorem] (supposing a~priori [that $m = p$]). Accordingly, [his theorem] may be considered as a generalization of Babbage's criterion for primality.''

Here we offer another generalization of Babbage's primality test.

\begin{theorem}[Least-Prime-Factor Test] \label{THM:LPF}
The least prime factor of an integer $m>1$ is the smallest natural number $\ell$ satisfying
	\begin{align} \label{EQ:LPF2}
\binom{m+\ell}{\ell}  \not \equiv 1 \pmod{m}.
\end{align}
For that value of $\ell,$ the least non-negative residue of 
$\binom{m+\ell}{\ell}$ modulo $m$ is $\frac{m}{\ell} + 1.$
\end{theorem}
The proof is given in Section~\ref{SEC:proof}. 

Babbage's primality test is an easy corollary of the least-prime-factor test. Indeed, Theorem~\ref{THM:LPF} implies a sharp version of Theorem~\ref{THM:test} noticed by Granville \cite{granville} in $1995.$

\begin{corollary}[Sharp Babbage Primality Test] \label{THM:short}
Theorem~\ref{THM:test} remains true if the range for $n$ is shortened to $0 \le n\le \sqrt{p}.$
\end{corollary}

\begin{proof}
An integer $m>1$ is a prime if and only if its least prime factor $\ell$ exceeds
$\sqrt{m}.$ The corollary follows by setting $m=p$ in Theorem~\ref{THM:LPF}. \forceqed
\end{proof}

To see that {\em Corollary \ref{THM:short} is sharp in that the range for $n$ cannot be further shortened to} $0\le n \le \sqrt{p}-1$, let $q$ be any prime and set $p=q^2$. Then $p$ is not a prime, but the least-prime-factor test with $m=p$ and $\ell=q$ implies \eqref{EQ:test} when \mbox{$0\le n \le q-1$}.

\begin{problem}
Since the ``if'' part of Babbage's primality test is a consequence both of Me\v{s}trovi\'{c}'s theorem and of the least-prime-factor test, one may ask, {\em Is there a common generalization of Me\v{s}trovi\'{c}'s theorem and Theorem~\ref{THM:LPF}?}
(Note, though, that the modulus in the former is~$p,$ while that in the latter is~$m.$)
\end{problem}

Actually, the incongruence \eqref{EQ:LPF2} holds more generally 
if the {\em least} prime factor $\ell \mid m$ is replaced with {\em any} prime factor $p \mid m$. The following extension of the  least-prime-factor test is proven in Section~\ref{SEC:proof}. See also Sondow \cite[Part (a)]{sondow}.


\begin{theorem} \label{THM:pft}
$(i)$ Given a positive integer $m$ and a prime factor $p\mid m$, we have
\begin{equation} \label{EQ:primedivisor}
\binom{m+p}{p} \not\equiv 1 \pmod{m}.
\end{equation}
$(ii)$ If in addition $p^r\mid m$ but $p^{r+1}\nmid m$, where $r\ge1$, then
\begin{equation} \label{EQ:power}
\binom{m+p}{p} \equiv \frac{m}{p} + 1 \not\equiv 1 \pmod{p^r}.
\end{equation}
\end{theorem}

Part ($i$) is clearly equivalent to the statement that {\em if $d>1$ divides $m$ and $\binom{m+d}{d} \equiv 1 \pmod{m}$, then $d$ is composite.}
As an example, for $m=260$ and $d=10$ we have
$$ \binom{m+d}{d} = \binom{270}{10}  = 479322759878148681 \equiv 1 \pmod{260}.$$
The sequence of integers $m>1,$ for which some integer $d$ (necessarily composite) satisfies 
\begin{equation} \label{EQ:equiv 1}
d>1 , \qquad d \mid m, \qquad  \binom{m+d}{d} \equiv 1 \pmod{m},
\end{equation}
begins \cite[Seq. A290040]{oeis}
$$
m=260, 1056, 1060, 3460, 3905, 4428, 5000, 5060, 5512, 5860, 6372, 6596,\dotso
$$
and the sequence of smallest such divisors $d$ is, respectively, \cite[Seq. A290041]{oeis}
\begin{equation} \label{EQ:d}
d=10, 264, 10, 10, 55, 18, 20, 10, 52, 10, 18, 34,\dotso.
\end{equation}

\begin{problem}
Does Theorem~\ref{THM:pft} extend to prime power factors, i.e., does \eqref{EQ:primedivisor} also hold when $p$ is replaced with $p^k$, where $p^k\mid m$ and $k>1$? In particular, in the sequence \eqref{EQ:d}, is any term $d$ a prime power?
\end{problem}

See \cite[Part (c)]{sondow}.

Babbage also claimed a necessary and sufficient condition for primality based on a {\em single} congruence. But he proved only necessity, so we call it a test for non-primality.

\begin{theorem}[{\bf Babbage's Non-Primality Test}] \label{THM:bab}
An integer $m\ge3$ is composite if
\begin{equation} \label{EQ:babcong}
\binom{2m-1}{m-1} \not\equiv 1 \pmod{m^2}.
\end{equation}
\end{theorem}
Our version of his proof is given in Section~\ref{SEC:non-primality}.

Not only did Babbage not prove the claimed converse, but in fact it is false. Indeed,
{\em the numbers \mbox{$m_1=p_1^2=283686649$} and \mbox{$m_2=p_2^2=4514260853041$} are composite but do not satisfy} \eqref{EQ:babcong}, where \mbox{$p_1=16843$} and \mbox{$p_2=2124679$} are primes.
 
Here $p_1$ (indicated by Selfridge and Pollack in $1964$) and $p_2$ (discovered by Crandall, Ernvall and Mets\"{a}nkyl\"{a} in $1993$) are {\em Wolstenholme primes}, so called by Mcintosh \cite{mcintosh} because, while {\bf Wolstenholme's theorem} \cite{wolstenholme} (see \cite{granville,mestrovic,tw} and \cite[p.~73]{moll}) of $1862$ guarantees that {\em every prime $p\ge5$ satisfies}
\begin{equation} \label{EQ:Wolstenholme}
\binom{2p-1}{p-1} \equiv 1 \pmod{p^3},
\end{equation}
in fact $p_1$ and $p_2$ satisfy the congruence in \eqref{EQ:Wolstenholme} modulo $p^4$, not just $p^3$ (see Guy \cite[p.~131]{guy} and Ribenboim \cite[p.~23]{ribenboim}).

Note that \eqref{EQ:Wolstenholme} strengthens Babbage's non-primality test, as Theorem~\ref{THM:bab} is equivalent to the statement that {\em the congruence in \eqref{EQ:Wolstenholme} holds modulo $p^2$ for any prime} $p\ge3$. 

In their solutions to a problem by Segal in the \textit{Monthly}, Brinkmann \cite{sb} and Johnson \cite{sj} made Babbage's and Wolstenholme's theorems more precise by showing that {\em every prime $p\ge5$ satisfies the congruences}
$$\binom{2p-1}{p-1} \equiv 1-\frac23 p^3B_{p-3} \equiv \binom{2p^2-1}{p^2-1} \pmod{p^4},$$
where $B_k$ denotes the $k$th \textit{Bernoulli number}, a rational number. (See also Gardiner \cite{gardiner} and Mcintosh \cite{mcintosh}.) Thus, {\em a prime \mbox{$p\ge5$} is a Wolstenholme prime if and only if} $B_{p-3} \equiv 0 \pmod{p}$. (The congruence means that $p$~divides the numerator of $B_{p-3}$.)
In that case, the square of that prime, say \mbox{$m=p^2$}, is composite but must satisfy
$$\binom{2m-1}{m-1} \equiv 1\pmod{m^2},$$ thereby providing a counterexample to the converse of Babbage's non-primality test.

Johnson \cite{sj} commented that  ``interest in [Wolstenholme primes] arises from the fact that in $1857$, Kummer proved that the first case of [Fermat's Last Theorem] is true for all prime exponents $p$ such that $p\nmid B_{p-3}$.''

We have seen that  the converse of Babbage's non-primality test is false.
The converse of Wolstenholme's theorem is the statement that {\em if $p\ge5$ is composite, then \eqref{EQ:Wolstenholme} does not hold.}
It is not known whether this is generally true. 
A proof that it is true for {\em even} positive integers was outlined by Trevisan and Weber \cite{tw} in $2001$.
In Section \ref{SEC:non-primality}, we fill in some details omitted from their argument and extend it to prove the following stronger result.

\begin{theorem}[Converse of Babbage's Non-Primality Test for Even Numbers] \label{THM:even}
If a positive integer $m$ is even, then
\begin{equation} \label{EQ:even}
\binom{2m-1}{m-1} \not\equiv1\pmod{m^2}.
\end{equation}
\end{theorem}

\section{Proofs of the least-prime-factor test and its extension} \label{SEC:proof}
We prove Theorems~\ref{THM:LPF} and \ref{THM:pft}. The arguments
use only mathematics available in Babbage's time.

\begin{proof}[Theorem~\ref{THM:LPF}]
As $\ell$ is the smallest prime factor of $m,$ if $0<k < \ell$ then $k!$ and $m$ are coprime. In that case, {\bf B\'{e}zout's identity}
(proven in $1624$ by Bachet in a book with the charming title {\em Pleasant and Delectable Problems} \cite[p. 18, Proposition XVIII]{bachet}\textemdash see \cite[Section 4.3]{chabert}) gives integers $a$ and $b$ with 
\mbox{$ak!+bm=1$}. Multiplying B\'{e}zout's equation by the number \mbox{$\binom{m}{k} = m(m-1)\dotsb(m-k+1)/k!$} yields
$$ a m(m-1)\dotsb(m-k+1) + bm\binom{m}{k} = \binom{m}{k},$$
so $\binom{m}{k}  \equiv0\!\pmod{m}$ 
if $1 \le k \le \ell-1.$
Now, for \mbox{$n=0,1,\dotsc, \ell-1$}, {\bf Vandermonde's convolution} \cite{vandermonde} (see \cite[p.~164]{moll}) of $1772$ gives
\begin{equation*} \label{EQ:Chu} 
\binom{m+n}{n} = \sum_{k=0}^n \binom{m}{k} \binom{n}{n-k}
\equiv \binom{m}{0} \binom{n}{n}  \pmod{m}.
\end{equation*}
(To see the equality, equate the coefficients of $x^n$ in the expansions of \mbox{$(1+x)^{m+n}$} and $(1+x)^m(1+x)^n.$)
Thus, we arrive at the congruences
\begin{equation} \label{EQ:step1}
\binom{m+n}{n}\equiv 1 \pmod{m} \qquad (0 \le n \le \ell-1).
\end{equation}

On the other hand, from the identity
\begin{equation} \label{EQ:identity}
\binom{a}{b} = 
 \frac{a}{b} \binom{a-1}{b-1}
\end{equation}
(to prove it, use factorials), 
the congruence \eqref{EQ:step1} for $n=\ell-1$, the integrality of \mbox{$\frac{m+\ell}{\ell}=\frac{m}{\ell}+1$}, 
and the inequality $\ell>1$ (as $\ell$ is a prime), 
we deduce that
\begin{equation*} \label{EQ:lpftest}
 \binom{m+\ell}{\ell} = \frac{m+\ell}{\ell}\binom{m+\ell-1}{\ell-1}\equiv \frac{m}{\ell}+1  \not\equiv 1 \pmod{m}.
\end{equation*}
 Together with \eqref{EQ:step1}, this implies the least-prime-factor test. \forceqed
\end{proof}

\begin{proof}[Theorem~\ref{THM:pft}]
It suffices to prove (ii).
Set
$$g\DefEqDisp\gcd((p-1)!,m)  \qquad  \text{and}\qquad m_p\DefEqDisp\frac{m}{g}.$$
Note that
\begin{equation} \label{EQ:implies}
p  \ \text{prime} \implies p\nmid g \implies p^r\mid m_p,
\end{equation}
since $p^r\mid m$. B\'{e}zout's identity
gives integers $a$ and $b$ with
\mbox{$a(p-1)!+bm=g$}. When $0<k<p$, multiplying B\'{e}zout's equation by $\binom{m}{k}$
yields
$$ a m(m-1)\dotsb(m-k+1)\frac{(p-1)!}{k!} + bm\binom{m}{k} = g\binom{m}{k}$$
with $(p-1)!/k!$ an integer, so $g\binom{m}{k}  \equiv0\!\pmod{m}$. Dividing by $g$ gives
$$\binom{m}{k}  \equiv0\!\pmod{m_p}\quad (1\le k \le p-1).$$
Combining this with \eqref{EQ:identity} and Vandermonde's convolution, we get
\begin{align} \label{EQ:Chu2} 
\begin{split}
 \binom{m+p}{p} = \frac{m+p}{p}\binom{m+p-1}{p-1} &= \frac{m+p}{p}\sum_{k=0}^{p-1} \binom{m}{k} \binom{p-1}{p-1-k}\\
&\equiv  \frac{m}{p}+1 \pmod{m_p}.
\end{split}
\end{align}
As $p^{r+1}\nmid m$, we have $p^r\nmid \frac{m}{p}$. Now, \eqref{EQ:implies} and \eqref{EQ:Chu2} imply \eqref{EQ:power}, as required. \forceqed
\end{proof}

\section{Proofs of Babbage's non-primality test and its converse for even numbers} \label{SEC:non-primality}

The following proof is close to the one Babbage gave.

\begin{proof}[Theorem \ref{THM:bab}]
Suppose on the contrary that $m$ is prime. If we have $1 \le n \le m-1$, then $m$ divides the numerator of $\binom{m}{n} = m!/n!(m-n)!$ but not the denominator, so
\mbox{$\binom{m}{n}\equiv 0 \pmod{m}$}.
Thus, by 
\eqref{EQ:identity} and a famous case of Vandermonde's convolution,
\begin{equation} \label{EQ:special}
2\binom{2m-1}{m-1} =\binom{2m}{m} = \sum_{n=0}^m\binom{m}{n}^2 \equiv 1^2+1^2 \equiv 2 \pmod{m^2}.
\end{equation}
But as $m\ge3$ is odd, \eqref{EQ:special} contradicts
\eqref{EQ:babcong}. Therefore, $m$ is composite. \forceqed
 \end{proof}

Before giving the proof of Theorem~\ref{THM:even}, we establish two lemmas.
For any positive integer $k,$ let $2^{v(k)}$ denote the highest power of $2$ that divides $k.$

\begin{lemma} \label{LEM:power}
If $m\ge n\ge1$ are integers satisfying $n\le2^{v(m)},$ then the formula $v(\binom{m}{n}) = v(m)-v(n)$ holds.
\end{lemma}

\begin{proof}
Let $m=2^rm'$ with $m'$ odd. Note that $v(2^rm'-k) = v(k)$ if \mbox{$0<k < 2^r$}. 
({\em Proof.} Write $k =2^tk',$ where $0\le t = v(k) \le r-1$ and $k'$ is odd. Then \mbox{$2^{r-t}m'-k'$} is also odd, so $v(2^rm'-k) =v(2^t(2^{r-t}m'-k'))=t=v(k).$)
The logarithmic formula \mbox{$v(ab)=v(a)+v(b)$} then implies that when $1 \le n\le2^r$ the exponent of the highest power of $2$ that divides the product
$$n!\binom{m}{n}= 2^rm'(2^rm'-1)(2^rm'-2)\dotsb(2^rm'-(n-1))$$
is
$v(n!)+v(\binom{m}{n}) = r +v(1\cdot2\dotsb(n-1))$, so $v(\binom{m}{n}) = r-v(n)$. As \mbox{$r=v(m)$}, this proves the desired formula.  \forceqed
\end{proof}

Lemma \ref{LEM:power} is sharp in that the hypothesis $n\le2^{v(m)}$ cannot be replaced with the weaker hypothesis $v(n) \le v(m).$ For example, $v(\binom{10}{6}) = v(210) = 1$, but $v(10)-v(6) =0.$

\begin{lemma} \label{LEM:powerof2}
A binomial coefficient $\binom{2m-1}{m-1}$ is odd if and only if $m=2^r$ for some $r\ge0.$
\end{lemma}
\begin{proof} 
{\bf Kummer's theorem} \cite{kummer} (see \cite[p.~78]{moll} or \cite{pomerance}) for the prime~$2$ states that $v(\binom{a+b}{a})$ equals 
the number of carries when adding $a$ and $b$ in base~$2$ arithmetic. 
Hence $v(\binom{m+m}{m})$ is the number of ones in the binary expansion of $m$, and so \mbox{$v(\binom{2m}{m})=1$} if and only if $m=2^r$ for some $r\ge0$. 
As $\binom{2m}{m}  = 2 \binom{2m-1}{m-1}$ by \eqref{EQ:identity}, we are done.  \forceqed
\end{proof}

We can now prove the converse of Babbage's non-primality test for even numbers.

\begin{proof}[Theorem~\ref{THM:even}] 
For $m\ge2$ not a power of $2,$ Lemma~\ref{LEM:powerof2} implies that $\binom{2m-1}{m-1}$ is even, so $\binom{2m-1}{m-1}$ is congruent modulo~$4$ to either $0$ or $2$.
For $m\ge2$ a power of $2$, say $m=2^r$, the equalities in \eqref{EQ:special} and the symmetry $\binom{m}{n} =\binom{m}{m-n}$ yield
$$\binom{2m-1}{m-1} = 1 + \frac12\binom{2^r}{2^{r-1}}^2+ \sum_{k=1}^{2^{r-1}-1} \binom{2^r}{k}^2, $$
and Lemma \ref{LEM:power} implies that $\frac12\binom{2^r}{2^{r-1}}^2 \equiv2\!\!\pmod{4}$ 
and that $ \binom{2^r}{k}^2\equiv0\!\!\pmod{4}$ when \mbox{$0<k<2^{r-1}$}; thus, by addition $\binom{2m-1}{m-1} \equiv 3\!\! \pmod{4}$.
Hence for all $m\ge2$ we have $\binom{2m-1}{m-1} \not\equiv 1\!\! \pmod{4}$. Now as $4$ divides $m^2$ when $m$ is even, \eqref{EQ:even} holds a fortiori. 
 This completes the proof. \forceqed
\end{proof}


\printindex

\end{document}